\documentclass[12pt,leqno]{amsart}
\usepackage{times}
\usepackage{amssymb}
\usepackage[dvips]{graphicx}

\allowdisplaybreaks
\usepackage{amsmath,amssymb,txfonts}
\usepackage{amssymb}
\usepackage{amsxtra}
\usepackage{amsmath}
\usepackage{txfonts}

\newtheorem{thm}{Theorem}

\begin{document}

\title{Optimal geometric estimates for fractional Sobolev capacities}

\author{Jie Xiao}
\address{Department of Mathematics and Statistics, Memorial University, St. John's, NL A1C 5S7, Canada}
\email{jxiao@mun.ca}
\thanks{Research supported in part by: NSERC and URP of MUN, Canada.}






\begin{abstract}
This note develops certain sharp inequalities relating the fractional Sobolev capacity of a set to its standard volume and fractional perimeter.
\end{abstract}

\maketitle

Partially motivated by \cite{Xiao06, Xiao07}, this note discovers some optimal estimates linking the fractional Sobolev capacity of a set to its standard volume and fractional perimeter.

\subsection*{1. Fractional Sobolev capacities and their basic properties}

Let $0<\alpha<1$ and $C_0^\infty$ denote the class of all smooth functions with compact support in $\mathbb R^n$. Define the fractional Sobolev space (or the homogeneous $(\alpha,1,1)$-Besov space) $\dot{\Lambda}^{1,1}_\alpha$ as the completion of all functions $f\in
C^\infty_0$ with
$$
\|f\|_{\dot{\Lambda}_\alpha^{1,1}}=\int_{\mathbb R^n}\left(\int_{\mathbb R^n}{|f(x+h)-f(x)|}dx\right)\,\frac{dh}{|h|^{n+\alpha}}.
$$
Attached to $\dot{\Lambda}^{1,1}_\alpha$ is the following set-function:
$$
{\hbox{cap}}(K;
\dot{\Lambda}_\alpha^{1,1})=\inf\big\{\|f\|_{\dot{\Lambda}_\alpha^{1,1}}:\quad
f\in C^\infty_0\quad\&\quad f\ge 1_K\big\}\quad\forall\quad\hbox{compact}\quad K\subset\mathbb R^n.
$$
Here and henceforth, $1_E$ stands for the indicator of a set $E\subset\mathbb R^n$. This definition is extended to any set $E\subset\mathbb R^n$ via
$$
\hbox{cap}(E; \dot{\Lambda}_\alpha^{1,1})=\inf_{\hbox{open}\ O\supseteq E}\hbox{cap}(O;\dot{\Lambda}_\alpha^{1,1})=\inf_{\hbox{open}\ O\supseteq E}\left(\sup_{\hbox{compact}\ K\subseteq O}\hbox{cap}(K;\dot{\Lambda}_\alpha^{1,1})\right).
$$

The number $\hbox{cap}(E; \dot{\Lambda}_\alpha^{1,1})$ is called the fractional Sobolev capacity (or the homogeneous end-point Besov capacity) of $E$; see also \cite{Adams, Adams09, AdamsXiao, Xiao06, S}. Note that (cf. \cite{MS1, MS2, BBM1, BBM2, L1})
$$
\lim_{\alpha\to 0}\alpha \|f\|_{\dot{\Lambda}^{1,1}_\alpha}=2n\omega_n \|f\|_{L^1}\quad\&\quad \lim_{\alpha\to 1}(1-\alpha)\|f\|_{\dot{\Lambda}^{1,1}_\alpha}=\tau_n \|\nabla f\|_{L^1}
$$
where $\omega_n$ is the volume of the unit ball $\mathbb B^n$ of $\mathbb R^n$ and 
$$
\tau_n=\int_{\mathbb S^{n-1}}|\cos\theta|\,d\sigma
$$
with: $\mathbb S^{n-1}$ being the unit sphere of $\mathbb R^n$; $\theta$ being the angle deviation from the vertical direction; and $d\sigma$ being the standard area measure on $\mathbb S^{n-1}$. So, we have that for any compact $K\subset\mathbb R^n$,
$$
\lim_{\alpha\to 0}\alpha \hbox{cap}(K; \dot{\Lambda}_\alpha^{1,1})
=2n\omega_n V(K)\quad\&\quad \lim_{\alpha\to 1}(1-\alpha)\hbox{cap}(K; \dot{\Lambda}_\alpha^{1,1})=\tau_n \hbox{cap}(K; \dot{W}^{1,1})
$$
where
$$
V(K)=\int_K\,dx\quad\&\quad {\hbox{cap}}(K;
\dot{W}^{1,1})=\inf\big\{\|\nabla f\|_{L^1}:\quad
f\in C^\infty_0\quad\&\quad f\ge 1_K\big\}
$$
and
$$
\hbox{cap}(E; \dot{W}^{1,1})=\inf_{\hbox{open}\ O\supseteq E}\hbox{cap}(O;\dot{W}^{1,1})=\inf_{\hbox{open}\ O\supseteq E}\left(\sup_{\hbox{compact}\ K\subseteq O}\hbox{cap}(K;\dot{W}^{1,1})\right)\quad{\forall}\quad E\subset\mathbb R^n.
$$

\begin{thm}\label{t1} The nonnegative set-function $E\mapsto \hbox{cap}(E; \dot{\Lambda}_\alpha^{1,1})$ enjoys three essential properties below:

{\rm (i)} Homogeneity: $\hbox{cap}(rE; \dot{\Lambda}_\alpha^{1,1})=r^{n-\alpha}\hbox{cap}(E; \dot{\Lambda}_\alpha^{1,1}) \ \forall\  rE=\{rx:\ x\in E\}\subset\mathbb R^n\ \&\ r\in [0,\infty)$.

{\rm (ii)} Monotonicity: $E_1\subseteq E_2\subset\mathbb R^n\Rightarrow\hbox{cap}(E_1; \dot{\Lambda}_\alpha^{1,1})\le\hbox{cap}(E_2; \dot{\Lambda}_\alpha^{1,1})$.

{\rm (iii)} Upper-semi-continuity: $\lim_{j\to\infty}\hbox{cap}(K_j; \dot{\Lambda}_\alpha^{1,1})=\hbox{cap}(\cap_{j=1}^\infty K_j; \dot{\Lambda}_\alpha^{1,1})\ \forall$\ sequence\ $\{K_j\}_{j=1}^\infty$ of compact subsets of $\mathbb R^n$ with $K_1\supseteq K_2\supseteq K_3\supseteq\cdots$.
\end{thm}

\begin{proof} (i) follows from 
$$
\|f(r\cdot)\|_{\dot{\Lambda_\alpha^{1,1}}}=r^{\alpha-n}\|f\|_{\dot{\Lambda_\alpha^{1,1}}}\quad\forall\quad r\in [0,\infty).
$$
and the definition of $\hbox{cap}(\cdot,\dot{\Lambda_\alpha^{1,1}})$.

(ii) follows from the definition of $\hbox{cap}(\cdot,\dot{\Lambda_\alpha^{1,1}})$.

(iii) follows from a careful treatment. Suppose $\{K_j\}_{j=1}^\infty$ is a decreasing sequence of compact subsets of $\mathbb R^n$. Then
$K=\cap_{j=1}^\infty K_j$ is compact. Following the argument for \cite[Theorem 2.2(iv)]{HKM}, for any $\epsilon\in (0,1)$ there is a function $f\in C_0^\infty$ such that
$$
f\ge 1_{K}\quad\&\quad \|f\|_{\dot{\Lambda}_\alpha^{1,1}}
< \hbox{cap}(K; \dot{\Lambda}_\alpha^{1,1})+\epsilon.
$$
Note that if $j$ is sufficiently large then $K_j$ is contained in the compact set $\{x\in\mathbb R^n: f(x)\ge 1-\epsilon\}$. So, an application of (ii) and the definition of $\hbox{cap}(\cdot; \dot{\Lambda}_\alpha^{1,1})$ derives
$$
\lim_{j\to\infty}\hbox{cap}(K_j; \dot{\Lambda}_\alpha^{1,1})\le \hbox{cap}\big(\{x\in\mathbb R^n: f(x)\ge 1-\epsilon\}; \dot{\Lambda}_\alpha^{1,1}\big)\le\frac{\|f\|_{\dot{\Lambda}_\alpha^{1,1}}}{1-\epsilon}\le \frac{\hbox{cap}(K; \dot{\Lambda}_\alpha^{1,1})+\epsilon}{1-\epsilon}.
$$
Upon letting $\epsilon\to 0$ and using (ii) again, we get
$$
\hbox{cap}(K; \dot{\Lambda}_\alpha^{1,1})\le\lim_{j\to\infty}\hbox{cap}(K_j; \dot{\Lambda}_\alpha^{1,1})\le \hbox{cap}(K; \dot{\Lambda}_\alpha^{1,1}),
$$
as desired.
\end{proof}

For any set $E\subset\mathbb R^n$, let $E^c=\mathbb R^n\setminus E$ and compute 
$$
\|1_E\|_{\dot{\Lambda}_\alpha^{1,1}}=2\int_{E}\int_{E^c}\frac{dxdy}{|x-y|^{n+\alpha}}\equiv 2P_\alpha(E).
$$ 
whose half $P_\alpha(E)$ is called the fractional $\alpha$-perimeter; see e.g. \cite{FMM, L2}.
Notice that 
$$
\lim_{\alpha\to 0}\alpha P_\alpha(E)=n\omega_n V(E)\quad\&\quad \lim_{\alpha\to 1}(1-\alpha)P_\alpha(E)=2^{-1}\tau_n P(E)
$$
where $P(E)$ is the perimeter of $E$. So, we get an extension of \cite[Lemma 2.2.5]{Maz} from the limit $\alpha\to 1$ to the intermediate value $0<\alpha<1$ that connects the fractional Sobolev capacity and the fractional perimeter.

\begin{thm}
\label{t2} If $K$ is a compact subset of $\mathbb R^n$, then 
$$
\hbox{cap}(K; \dot{\Lambda}_\alpha^{1,1})=2\inf_{O\in\mathsf{O}^\infty(K)}P_\alpha(O)
$$ 
where $\mathsf{O}^\infty(K)$ denotes the class of all open sets with $C^\infty$  boundary that contain $K$.
\end{thm} 
\begin{proof} Given a compact $K\subset\mathbb R^n$. On the one hand, if $f\in C_0^\infty$ and $f\ge 1_K$ then 
$$
K\subset\{x\in\mathbb R^n: f(x)>t\}\quad\forall\quad t\in (0,1),
$$
and hence an application of the generalized co-area formula in \cite{Vi} (cf. \cite[Theorem 1.2]{Xiao06} for another version of the co-area formula of dimension $n-\alpha$) gives
$$
\|f\|_{\dot{\Lambda}_\alpha^{1,1}}=2\int_{0}^\infty P_\alpha\big(\{x\in\mathbb R^n: f(x)>t\}\big)\,dt\ge 2\int_0^1 P_\alpha\big(\{x\in\mathbb R^n: f(x)>t\}\big)\,dt\ge 2\inf_{O\in \mathsf{O}^\infty(K)}P_\alpha(O).
$$
This, along with the definition of $\hbox{cap}(K; \dot{\Lambda}_\alpha^{1,1})$, implies
$$
\hbox{cap}(K; \dot{\Lambda}_\alpha^{1,1})\ge 2\inf_{O\in \mathsf{O}^\infty(K)}P_\alpha(O).
$$

On the other hand, according to Theorem \ref{t1}(ii) and \cite[Theorem 3.1]{HV} we have 
$$
\hbox{cap}(K; \dot{\Lambda}_\alpha^{1,1})\le \hbox{cap}(\overline{O}; \dot{\Lambda}_\alpha^{1,1})\le 2P_\alpha(O)\quad\forall\quad O\in \mathsf{O}^\infty(K),
$$
whence 
$$
\hbox{cap}(K; \dot{\Lambda}_\alpha^{1,1})\le 2\inf_{O\in \mathsf{O}^\infty(K)}P_\alpha(O).
$$
Therefore, the desired formula for $\hbox{cap}(K; \dot{\Lambda}_\alpha^{1,1})$ follows.
\end{proof}

As an immediate consequence of Theorem \ref{t2}, we have
$$
\lim_{\alpha\to 0}\alpha \hbox{cap}(K; \dot{\Lambda}_\alpha^{1,1})
=2n\omega_n V(K)\quad\&\quad \lim_{\alpha\to 1}(1-\alpha)\hbox{cap}(K; \dot{\Lambda}_\alpha^{1,1})=\tau_n P(K)\quad\forall\ \ \hbox{compact}\ \ K\subset\mathbb R^n.
$$
\subsection*{2. Fractional Sobolev inequalities and their geometric forms} The next analytic-geometric assertion indicates that the fractional Sobolev capacity plays a decisive role in improving the fractional isoperimetric inequality \cite[(4.2)]{FS}.

\begin{thm}
\label{t3} Let $
\kappa_{n,\alpha}={\omega_n^\frac{n-\alpha}{n}}{\big(2P_\alpha(\mathbb B^n)\big)^{-1}}.$ Then:

{\rm(i)} The analytic inequality
\begin{equation}\label{eq1}
\|f\|_{L^\frac{n}{n-\alpha}}\le \kappa_{n,\alpha}\left(\int_0^\infty \big(\hbox{cap}(\{x\in\mathbb R^n: |f(x)|\ge t\}; \dot{\Lambda}_\alpha^{1,1})\big)^\frac{n}{n-\alpha}\,dt^\frac{n}{n-\alpha}\right)^\frac{n-\alpha}{n}\ \forall\ f\in C_0^\infty
\end{equation}
is equivalent to the geometric inequality
\begin{equation}\label{eq2}
\big(V(O)\big)^\frac{n-\alpha}{n}\le \kappa_{n,\alpha}\, \hbox{cap}(\overline{O};\dot{\Lambda}_\alpha^{1,1})\quad \forall\ \hbox{bounded\ domain}\ \ O\subset\mathbb R^n\ \hbox{with}\ C^\infty\ \hbox{boundary}\ \partial O.
\end{equation}
Moreover, both $(\ref{eq1})$ and $(\ref{eq2})$ are true and sharp.

{\rm(ii)} The analytic inequality
\begin{equation}\label{eq3}
\left(\int_0^\infty \big(\hbox{cap}(\{x\in\mathbb R^n: |f(x)|\ge t\}; \dot{\Lambda}_\alpha^{1,1})\big)^\frac{n}{n-\alpha}\,dt^\frac{n}{n-\alpha}\right)^\frac{n-\alpha}{n}\le\|f\|_{\dot{\Lambda}^{1,1}_\alpha}\ \forall\ f\in C_0^\infty
\end{equation}
is equivalent to the geometric inequality
\begin{equation}\label{eq4}
\hbox{cap}(\overline{O};\dot{\Lambda}_\alpha^{1,1})\le 2 P_\alpha({O})\quad \forall\ \hbox{bounded\ domain}\ \ O\subset\mathbb R^n\ \hbox{with}\ C^\infty\ \hbox{boundary}\ \ \partial O.
\end{equation}
Moreover, both (\ref{eq3}) and (\ref{eq4}) are true and sharp.
\end{thm}

\begin{proof} (i) Suppose (\ref{eq2}) is valid. For any $C_0^\infty$ function $f$, set
$$
O_t(f)=\{x\in\mathbb R^n:\ |f(x)|> t\}\quad\forall\quad t\ge 0.
$$
Then an application of (\ref{eq2}) to $O_t(f)$ yields
$$
\|f\|_{\frac{n}{n-s}}=\left(\int_0^\infty V\big(O_t(f)\big)\,dt^ \frac{n}{n-\alpha}\right)^\frac{n-\alpha}{n}\le\kappa_{n,\alpha}\left(\int_0^\infty\Big(\hbox{cap}\big(\overline{O_t(f)};\dot{\Lambda}_\alpha^{1,1})\big)\Big)^\frac{n}{n-\alpha}\,dt^\frac{n}{n-\alpha}\right)^\frac{n-\alpha}{n},
$$
deriving (\ref{eq1}). Conversely, suppose (\ref{eq1}) is valid. For any bounded domain $O\subset\mathbb R^n$ with $C^\infty$ boundary $\partial O$, the Euclidean distance $\hbox{dist}(x,E)$ of a point $x$ to a set $E$, and $0<\epsilon<1$, let
$$
f_\epsilon(x)=\begin{cases}
1-\epsilon^{-1}\hbox{dist}(x,\overline{O})\quad\hbox{as}\quad \hbox{dist}(x,\overline{O})<\epsilon\\
0\quad\hbox{as}\quad \hbox{dist}(x, \overline{O})\ge\epsilon.
\end{cases}
$$
Then the inequality in (\ref{eq1}) is true for $f_\epsilon$. Consequently, via setting 
$$
O_\epsilon=\{x\in\mathbb R^n:\ \hbox{dist}(x,\overline{O})<\epsilon\}\quad\&\quad
\epsilon\to 0
$$ 
and using Theorem \ref{t1}(iii), we gain
\begin{align*}
&\big(V(O)\big)^\frac{n-\alpha}{n}\\
&\leftarrow\|f_\epsilon\|_{L^{\frac{n}{n-\alpha}}}\\
&\le\kappa_{n,\alpha}
\left(\int_0^\infty \big(\hbox{cap}(\overline{O_t(f_\epsilon)}; \dot{\Lambda}_\alpha^{1,1})\big)^\frac{n}{n-\alpha}\,dt^\frac{n}{n-\alpha}\right)^\frac{n-\alpha}{n}\\
&=\kappa_{n,\alpha}
\left(\int_0^1 \big(\hbox{cap}(\overline{O_t(f_\epsilon)}; \dot{\Lambda}_\alpha^{1,1})\big)^\frac{n}{n-\alpha}\,dt^\frac{n}{n-\alpha}\right)^\frac{n-\alpha}{n}\\
&\le\kappa_{n,\alpha}\hbox{cap}(\overline{O_\epsilon};\dot{\Lambda}_\alpha^{1,1})\\
&\rightarrow \kappa_{n,\alpha} \hbox{cap}(\overline{O}; \dot{\Lambda}_\alpha^{1,1}).
\end{align*}
This proves (\ref{eq2}). Moreover, the truth and the sharpness of (\ref{eq2}) (and hence (\ref{eq1}) via the just-checked equivalence) follow from the definition of $\hbox{cap}(O; \dot{\Lambda}_\alpha^{1,1})$ and the sharp fractional Sobolev inequality on \cite[Theorem 4.1: $p=1$]{FS}:
$$
\|f\|_{L^\frac{n}{n-\alpha}}\le\kappa_{n,\alpha}\|f\|_{\dot{\Lambda}^{1,1}_\alpha}
\quad \forall\ \ f\in C_0^\infty.
$$

(ii) Suppose (\ref{eq3}) is valid. Given $\epsilon\in (0,1)$ and a bounded domain $O\subset\mathbb R^n$ with $C^\infty$ boundary $\partial O$, select again $O_\epsilon\ \&\ f_\epsilon$ as above. Then Theorem \ref{t1}(ii) and the argument for \cite[Lemma 3.2]{HV} give
$$
\hbox{cap}(\overline{O}; \dot{\Lambda}_\alpha^{1,1})\le\left(\int_0^1 \Big(\hbox{cap}\big(\overline{O_t(f_\epsilon)}; \dot{\Lambda}_\alpha^{1,1}\big)\Big)^\frac{n}{n-\alpha}\,dt^\frac{n}{n-\alpha}\right)^\frac{n-\alpha}{n}\\
\le \|f_\epsilon\|_{\dot{\Lambda}^{1,1}_\alpha}\rightarrow \|1_O\|_{\dot{\Lambda}^{1,1}_\alpha}=2P_\alpha(O).
$$
In other words, (\ref{eq4}) is true. Conversely, suppose  (\ref{eq4}) is valid. Upon noticing that for any $C^\infty_0$ function $f$ with $O_t(f)$ being as above, the function 
$t\mapsto \hbox{cap}\big(\overline{O_t(f)};\ \dot{\Lambda}^{1,1}_\alpha\big)$
decreases on $[0,\infty)$ (thanks to Theorem \ref{t1}(ii)), we have
\begin{align*}
&t^{\frac{\alpha}{n-\alpha}}\Big(\hbox{cap}\big(\overline{O_t(f)}; \dot{\Lambda}^{1,1}_\alpha\big)\Big)^\frac{n}{n-\alpha}\\
&=\Big(t \hbox{cap}\big(\overline{O_t(f)};\ \dot{\Lambda}^{1,1}_\alpha\big)\Big)^{\frac{\alpha}{n-\alpha}}\hbox{cap}\big(\overline{O_t(f)};\ \dot{\Lambda}^{1,1}_\alpha\big)\\
&\le \left(\int_0^t\hbox{cap}\big(\overline{O_s(f)};\ \dot{\Lambda}^{1,1}_\alpha\big)\,ds\right)^{\frac{\alpha}{n-\alpha}}\hbox{cap}\big(\overline{O_t(f)};\ \dot{\Lambda}^{1,1}_\alpha\big)\\
&=\Big(\frac{n-\alpha}{n}\Big)\frac{d}{dt}\left(\int_0^t \hbox{cap}\big(\overline{O_s(f)};\ \dot{\Lambda}^{1,1}_\alpha\big)\,ds\right)^{\frac{n}{n-\alpha}},
\end{align*}
whence finding, along with Theorem \ref{t1}(ii), (\ref{eq4}), Theorem \ref{t2} and the previously-cited co-area formula,
$$
\left(\int_0^\infty \Big(\hbox{cap}\big(\overline{O_t(f)}\big);\ \dot{\Lambda}^{1,1}_\alpha\Big)^{\frac{n}{n-\alpha}}\,dt^\frac n{n-\alpha}\right)^\frac{n-\alpha}{n}\le\int_0^\infty\hbox{cap}\big(\overline{O_t(f)};\dot{\Lambda}^{1,1}_\alpha\big)\,dt\le 2\int_0^\infty P_\alpha\big(O_t(f)\big)\,dt=\|f\|_{\dot{\Lambda}_\alpha^{1,1}}.
$$
So, (\ref{eq3}) holds. Moreover, the truth of (\ref{eq4}) (and hence (\ref{eq3}) via the above equivalence) follows from Theorem \ref{t2}. In fact, if (\ref{eq4}) were not sharp, then an application of (\ref{eq2}) would derive that the sharp fractional isoperimetric inequality (cf. \cite[(4.2)]{FS})
$$
\big(V(E)\big)^\frac{n-\alpha}{n}\le 2\kappa_{n,\alpha}P_\alpha(E)\quad\forall\quad E\subset\mathbb R^n
$$
is not sharp, thereby reaching a contradiction. Thus, (\ref{eq4}) is sharp, and so is (\ref{eq3}).
\end{proof}

Theorem \ref{t3} comes actually from splitting both the sharp fractional Sobolev inequality and the fractional isoperimetric inequality whose equivalence (optimizing \cite[Theorem 1.1]{HV} under $G=\mathbb R^n$) is described below.

\begin{thm}\label{t4} The following three optimal statements are equivalent:

{\rm (i)} The fractional Sobolev inequality $\|f\|_{L^\frac{n}{n-\alpha}}\le \kappa_{n,\alpha}\|f\|_{\dot{\Lambda}_{\alpha}^{1,1}}$ holds for any $f\in C_0^\infty$.

{\rm (ii)} The fractional isocapacitary inequality $\big(V(O)\big)^\frac{n-\alpha}{n}\le \kappa_{n,\alpha}\hbox{cap}(\overline{O},\dot{\Lambda}_{\alpha}^{1,1})$ holds for any bounded domain $O\subset\mathbb R^n$ with $C^\infty$ boundary $\partial O$.

{\rm (iii)} The fractional isoperimetric inequality $\big(V(O)\big)^\frac{n-\alpha}{n}\le 2\kappa_{n,\alpha}P_\alpha(O)$ holds for any bounded domain $O\subset\mathbb R^n$ with $C^\infty$ boundary $\partial O$.

\end{thm}

\begin{proof} (i)$\Rightarrow$(ii) follows from the definition of $\hbox{cap}(\overline{O},\dot{\Lambda}_{\alpha}^{1,1})$. (ii)$\Rightarrow$(iii) follows from Theorem \ref{t2}. (iii)$\Rightarrow$(i) follows from the idea verifying (\ref{eq4})$\Rightarrow$(\ref{eq3}). As a matter of fact, assume (iii) is true. Given a function $f\in C_0^\infty$ with $O_t(f)$ being the same as in the proof of Theorem \ref{t3}. Obviously, $t\mapsto V(O_t(f))$ is a decreasing function on $[0,\infty)$. This monotonicity, together with the layer-cake formula, the chain rule, (iii) for $O_t(f)$, and the above-used co-area formula, derives
\begin{align*}
&\|f\|_{L^\frac{n}{n-\alpha}}\\
&=\left(\int_0^\infty V\big(O_t(f)\big)\,dt^\frac{n}{n-\alpha}\right)^\frac{n-\alpha}{n}\\
&=\int_0^\infty\frac{d}{dt}\left(\int_0^t V\big(O_s(f)\big)\,ds^\frac{n}{n-\alpha}\right)^\frac{n-\alpha}{n}\,dt\\
&=\int_0^\infty\left(\int_0^t V\big(O_s(f)\big)\,ds^\frac{n}{n-\alpha}\right)^{-\frac{\alpha}{n}}V\big(O_t(f)\big)t^\frac{\alpha}{n-\alpha}\,dt\\
&\le\int_0^\infty \Big(V\big(O_t(f)\big)\Big)^\frac{n-\alpha}{n}\,dt\\
&\le 2\kappa_{n,\alpha}\int_0^\infty P_\alpha\big(O_t(f)\big)\,dt\\
&=\kappa_{n,\alpha}\|f\|_{\dot{\Lambda}^{1,1}_\alpha},
\end{align*}
whence reaching (i).

\end{proof}

Note that for any compact set $K\subset\mathbb R^n$ and any bounded domain $O\subset\mathbb R^n$ one has
$$
\begin{cases}
\lim_{\alpha\to 1}(1-\alpha)^{-1}\kappa_{n,\alpha}=(n\omega_n^\frac{1}{n}\tau_n)^{-1}\\
\lim_{\alpha\to 1}(1-\alpha)\hbox{cap}\big(K;\ \dot{\Lambda}^{1,1}_\alpha\big)=\tau_n \hbox{cap}\big(K;\ \dot{W}^{1,1}\big)=\tau_n\inf_{\tilde{O}\in \mathsf{O}^\infty(K)}P(\tilde{O})\\
\lim_{\alpha\to 1}(1-\alpha)P_\alpha(O)=2^{-1}\tau_n P(O).
\end{cases}
$$
So, the limiting cases of Theorem \ref{t3} and Theorem \ref{t4} as $\alpha\to 1$ reduce to \cite[Theorems 1.1-1.2]{Xiao07} and \cite[Proposition 3.1]{Xiao09} plus the well-known Federer-Felming-Maz'ya equivalence between the isoperimetric inequality and the Sobolev inequality (cf. \cite{FF, M}), respectively.

\bibliographystyle{amsplain}

\end{document}